\documentclass[11pt]{amsart}
\usepackage{latexsym}
\usepackage{amssymb}
\usepackage{amsfonts}
\usepackage{graphicx}
\usepackage{epsf}
\usepackage[all]{xypic}

\usepackage{delarray}
\usepackage{setspace}
\newtheorem{theorem}{Theorem}[section]
\newtheorem{corollary}[theorem]{Corollary}

\newtheorem{proposition}[theorem]{Proposition}
\newtheorem{definition}[theorem]{Definition}
\newtheorem{example}[theorem]{Example}
\newtheorem{examples}[theorem]{Examples}

\newcommand{\End}{{\rm End}}

\def\bea{\begin{eqnarray*}}
\def\eea{\end{eqnarray*}}

\definemorphism{dato}\dashed \tip \notip
\definemorphism{Dato}\Dashed \tip \notip

\newcommand{\ra}{\rightarrow}

\def\smashco{\mathrel>\joinrel\mathrel\triangleleft}

\newcommand{\Mm}{\mathcal{M}}

\def\CC{{\mathbb C}}
\def\KK{{\mathbb K}}
\def\ZZ{{\mathbb Z}}

\begin{document}

\sloppy

\title[Semiperfect and coreflexive coalgebras]{Semiperfect and coreflexive coalgebras}


\begin{abstract}
We study non-counital coalgebras and their dual non-unital
algebras, and introduce the finite dual of a non-unital algebra.
We show that a theory that parallels in good part the duality in
the unital case can be constructed. Using this, we introduce a new
notion of left coreflexivity for counital coalgebras, namely, a
coalgebra is left coreflexive if $C$ is isomorphic canonically to
the finite dual of its left rational dual $Rat(_{C^*}C^*)$. We
show that right semiperfectness for coalgebras is in fact
essentially equivalent to this left reflexivity condition, and we
give the connection to usual coreflexivity. As application, we
give a generalization of some recent results connecting dual
objects such as quiver or incidence algebras and coalgebras, and
show that Hopf algebras with non-zero integrals (compact quantum groups) are coreflexive.
\end{abstract}

\author{S.D\u{a}sc\u{a}lescu${}^1$ and M.C. Iovanov${}^{1,2}$}
\address{${}^1$University of Bucharest, Facultatea de Matematica\\ Str.
Academiei 14, Bucharest 1, RO-010014, Romania}
\address{${}^2$University of Iowa\\
MacLean Hall, Iowa City, Iowa 52246, USA}
\address{e-mail: sdascal@fmi.unibuc.ro, iovanov@usc.edu, yovanov@gmail.com}

\date{}

\maketitle

\section{Introduction and Preliminaries}

Let $\KK$ be a field and let $\Gamma$ be a quiver. Two
combinatorial objects that are important in representation theory
are associated with $\Gamma$: the quiver algebra $\KK [\Gamma]$,
and the path coalgebra $\KK\Gamma$. The quiver algebra is an
algebra with enough idempotents, but does not have a unit unless
$\Gamma$ is finite, while the path coalgebra $\KK\Gamma$ is a
coalgebra with counit. In \cite{DIN1} it was investigated how one
of the two objects can be recovered from the other one; a finite
dual coalgebra $A^0$ was constructed for an algebra $A$ with
enough idempotents, and it was proved that $\KK \Gamma$ is
isomorphic to $\KK [\Gamma]^0$ provided that $\Gamma$ has no
oriented cycles and between any two vertices of $\Gamma$ there are
finitely many arrows. On the other hand, the algebra $\KK
[\Gamma]$ can be recovered as the left (or right) rational part of
$(\KK\Gamma)^*$, provided that for any vertex $v$ of $\Gamma$
there are finitely many paths starting at $v$ and finitely many
paths ending at $v$ (see \cite{DIN1}). A consequence of these
results is that $\KK \Gamma$ is coalgebra isomorphic to
$((\KK\Gamma)^{*rat})^0$ for certain general enough $\Gamma$.
Parallel results are obtained for incidence (co)algebras of a
locally finite partially ordered set $X$, in which case the role
of $\KK \Gamma, (\KK\Gamma)^*$ and $\KK[\Gamma]$ is played by the
incidence coalgebra $\KK X$, the incidence algebra $IA(X)$, and
the finite incidence algebra $FIA(X)$, consisting of functions in $IA(X)$ of
finite support. It is natural to ask whether a general context
that unifies these structures and results exists, for arbitrary
algebras and coalgebras.

The first step in obtaining such a context is having a well
behaved duality for non-unital algebras and coalgebras, that would
naturally include the (co)unital case. We extend the construction
of the finite dual coalgebra $A^0$ to the case of an arbitrary
algebra $A$, not necessarily with enough idempotents, and extend
the coalgebra theory to coalgebras without counit. We show that
several results on coalgebras, including the fundamental theorem
of coalgebras, carry over to non-counital coalgebras. These are
needed and used for the above mentioned unification, but may also
present some interest in their own. Some of the results from
coalgebra theory can be extended to the non-counital case with
parallel proofs, but other need new arguments avoiding the use of
the counit property. We explain how some results in the
non-counital case can be obtained by using a counitalization
construction for coalgebras, and the transfer of properties
through it. In Section 3 we construct $A^0$ for an arbitrary
algebra $A$, and we give equivalent characterizations for it. As
in the unital case, we show that locally finite representations of
$A$ are just corepresentations over its finite dual $A^0$.

(Co)reflexivity for algebras and coalgebras was introduced and
studied by R. Heyneman, D.E. Radford and E.Taft
(\cite{HR,Rad,T1,T2}) in tight connection to  the duality between
algebras and coalgebras. A coalgebra $C$ with counit is called
coreflexive if the canonical map $C\rightarrow (C^*)^0$ is
bijective. In other words, the coalgebra is recovered from its
dual algebra via the finite dual construction, i.e. it is its own
double dual. Coreflexivity is a finiteness condition, since it is
shown to be equivalent to every finite dimensional left
(equivalently, right) $C^*$-module being rational. Motivated by
finding a unifying context for the above mentioned combinatorial
results connecting the algebra and coalgebra structures associated
with quivers and partially ordered sets, in Section 4 we consider
another coreflexivity-type condition for a coalgebra $C$ with
counit. Instead of looking at the dual algebra of $C$, we look at
the rational left (or right) dual of $C$, regarded as a non-unital algebra
$Rat({}_{C^*}C^*)$ (or $Rat(C^*_{C^*})$), and consider the natural
map
 $\phi_l:C\rightarrow (Rat(_{C^*}C^*))^0$. We call a coalgebra $C$ \emph{left coreflexive} if
$\phi_l$ is an isomorphism. We first note a connection with
another well studied property for coalgebras, namely, the morphism
$\phi_l$ is injective if and only if $C$ is right semiperfect.
Recall that a coalgebra is right semiperfect if the category of
right comodules has enough projective objects, or equivalently,
indecomposable injective left comodules are  finite dimensional.
Moreover, if $C$ is right semiperfect and the coradical
 $C_0$ is coreflexive, we prove that
$\phi_l$ is an isomorphism. In this case, we show that $C$ is also
coreflexive in the sense of Radford and Taft. This situation
appears in \cite{DIN1}, for the case of coalgebras associated to
certain general quivers or PO-sets. We note that the condition
that $C_0$ is coreflexive is not very restrictive, and in fact
$C_0$ is essentially always coreflexive if $\KK$ is infinite.
Thus, this effectively gives an unexpected interpretation of
semiperfect coalgebras, as equivalent to a coreflexivity (or self
double dual) type of condition. Finally, we give an example
showing that a coreflexive coalgebra is not necessarily left (nor
right) coreflexive.

If $C$ is left and right semiperfect, we show that $\phi_l$ is
always bijective (independent of the supplementary assumption on
$C_0$), thus unifying and extending the results about path
coalgebras and incidence coalgebras mentioned above.

We also give an application and interpretation of our results as a
a double dual property of Hopf algebras with non-zero integrals;
many interesting quantum groups are such Hopf algebras. There are
many well known duality pairings for classes of Hopf algebras and
quantum groups. Let $G$ be an algebraic group over $\CC$, and let
$L$ be its Lie algebra. Let $H$ be the Hopf algebra of functions
on $G$. If $G$ is simply connected, then $H$ is the finite dual of
$U(L)$, i.e. $H\cong U(L)^0$. This duality is also preserved at
the quantum level, for example, there is an isomorphism
$SL_n(q)\cong U(sl_n(q))^0$ (\cite{K}). The group $G$ can be
recovered as the group $G(H^0)$ of grouplike elements in $H^0$,
and, in fact, we again have an isomorphism $H^0\cong
\CC[G]\smashco U(L)$, a smash product, by the theorems of Cartier,
Gabriel, Kostant, Milnor and Moore (see also \cite[Remark
3.7.3]{C}). For example, when $G$ is the additive algebraic group
$(\CC,+)$, then $H=\CC[X]$, and $H^0=\CC[G]\smashco U(L)$, where
$L$ is the abelian Lie algebra of dimension $1$. This example is
considered also in \cite[page 1123]{T1} (for countable
algebraically closed fields), but one sees that the coalgebra
(Hopf algebra) $H^0$ for $H=\CC[X]$ here is coreflexive (see
\cite{T1}).

It is natural to ask if there is some reflexivity (duality) property of
any of these classes of algebras. 
One motivation is the duality between commutative Hopf algebras
and algebraic groups, as noted above: $G\longmapsto H=O(G)$ =
algebra of functions on $G$, $H\longmapsto G=G(H^0)$. In fact the
above mentioned example of coreflexive Hopf algebra in \cite{T1}
holds more generally for every finite dimensional Lie algebra $L$.
Namely, using results in \cite{HR,I2}, it is not difficult to see
that if $L$ is a finite dimensional (say over $\CC$), $U(L)$ is
coreflexive. Moreover, with notations for $G$ and $H$ as above, we
see that  $H^0=\bigoplus\limits_{G}U(L)$ is a coreflexive
coalgebra too (see Proposition \ref{p.algr}).

We show that there is such a reflexivity for Hopf algebras with
nonzero integral, which are a generalization of algebras of
functions on compact groups (i.e. compact quantum groups). Namely,
we observe that such Hopf algebras are coreflexive in all the
possible meanings, i.e. both in the sense of Lin, Radford and
Taft, provided that the set of simple subcoalgebras satisfy a
certain non-restrictive condition (which is always true over
infinite fields), as well as with respect to our new reflexivity
notion.


We work over a field $\KK$. By algebra we mean an associative
algebra, not necessarily with unit. For such an algebra $A$ we
denote by $A-Mod$ the category of left $A$-modules (with no unital
condition, not even in the case where $A$ has a unit). If $A$ is
an algebra with unit, we denote by $_A{\mathcal M}$ the category
of unital left $A$-modules. By coalgebra we mean a coassociative
coalgebra which does not necessarily have a counit. If $C$ is such
a coalgebra, with comultiplication $\delta$, then a right
$C$-comodule is a space $M$ together a linear map
$\rho:M\rightarrow M\otimes C$ such that $(\rho \otimes
I)\rho=(I\otimes \delta)\rho$. The category of right $C$-comodules
is denoted by $Comod-C$. If $C$ has a counit, then the category of
counital $C$-comodules (i.e. satisfying the counit property) is
denoted by $\mathcal{M}^C$. For basic terminology and notation
about coalgebras and comodules we refer to \cite{DNR}, \cite{mo},
and \cite{Sw}.

\section{Non-counital coalgebras}

 If $A$ is an algebra, then the unitalization of $A$ is the
 algebra $A^1$
 with identity, obtained by adjoining to $A$ a unit element $u$, i.e.
 $A^1=A\oplus \KK u$, with multiplication defined by $(a+\alpha
 u)(b+\beta u)=ab+\alpha b+\beta a +\alpha\beta u$ for any $a,b\in
 A$ and $\alpha,\beta\in \KK$. Then the inclusion map
 $i:A\rightarrow A^1$ is an algebra map satisfying the following universal property. If $B$
 is an algebra with unit, then any algebra map $f:A\rightarrow B$
 extends uniquely to a map $F:A^1\rightarrow B$ of algebras with
 unit. The categories $A-Mod$ and $_{A^1}{\mathcal M}$ are isomorphic.
 Indeed, a left $A$-module $M$ can be viewed as a unital left
 $A^1$-module by extending the action of $A$ with the condition
 that $u$ acts as identity on $M$, while a unital left $A^1$-module is
 a left $A$-module by restricting scalars via $i$. Moreover, the
 lattices $\mathcal{L}_{A-Mod}(M)$ and $\mathcal{L}_{_A\mathcal{M}}(M)$ of subobjects of
 $M$ in these two categories coincide. We also note that
 every finite dimensional algebra $A$
embeds in a matrix algebra. Indeed, if $dim(A)=n$, then
 $\pi:A\rightarrow {\rm End}(A^1)\cong M_{n+1}(\KK)$, $\pi (a)(z)=az$ for any $a\in A$, $z\in A^1$,
 is an injective morphism of algebras.

We show that there are dual constructions and results for
coalgebras.

\begin{proposition}\label{l.2}
(i) Every coalgebra is a quotient of a coalgebra with counit.\\
(ii) Every finite dimensional coalgebra $(C,\delta)$ is a quotient
of a comatrix coalgebra. Consequently, there are elements
$(c_{ij})_{i,j=1,\dots,n}$ which span $C$ such that
$\delta(c_{ij})=\sum\limits_{k=1}^nc_{ik}\otimes c_{kj}$.
\end{proposition}
\begin{proof}
(i) Let $(C,\delta)$ be a coalgebra, and write $\delta (c)=\sum
c_1\otimes c_2$. Let $C^1=C\oplus \KK e$, the direct sum of $C$
and a 1-dimensional space. Then it is easy to check that $C^1$ has
a structure of a coalgebra with counit with comultiplication
$\Delta$  defined by
$$\Delta (c+\alpha
e)=c\otimes e+e\otimes c+\alpha e\otimes e+\sum c_1\otimes c_2$$
and counit $\epsilon$, $\epsilon (c+\alpha e)=\alpha$ for any
$c\in C, \alpha \in \KK$. Moreover, the map
$\pi:C^1\rightarrow C$ defined by $\pi (c+\alpha e)=c$, is a surjective morphism of coalgebras.\\
(ii) The dual algebra $C^*$ embeds in a matrix algebra $M_n(\KK)$,
and then $C$ is a quotient of the matrix coalgebra $M_n^c(\KK)$.
For the last part, if $\theta:M_n^c(\KK)\rightarrow C$ is a
surjective morphism of coalgebras and $e_{ij}$ is a comatrix basis
of $M_n^c(\KK)$, then it suffices to take $c_{ij}=\theta(e_{ij})$.
\end{proof}

The following describes basic properties of the construction
$C\mapsto C^1$, the "counitalization" of a coalgebra.

\begin{proposition}\label{l.3}
Let $(C,\delta)$ be a coalgebra, and let $(C^1,\Delta, \epsilon)$
be the coalgebra with counit constructed in Proposition \ref{l.2},
with the coalgebra map $\pi:C^1\rightarrow C$. The following
assertions are true.\\
(i) For any coalgebra with counit $D$ and any coalgebra morphism
$f:D\rightarrow C$, there exists a unique morphism of coalgebras
with counit $\overline{f}:D\rightarrow C^1$ such that
$\pi\overline{f}=f$. Thus the
counitalization functor $(-)^1$ is a right adjoint to the forgetful functor from the category of counital coalgebras to the category of coalgebras.\\
(ii) The category $Comod-C$ of right $C$-comodules is isomorphic
to the category  $\mathcal{M}^{C^1}$-comodules via the
corestriction of scalars functor associated to $\pi$. Moreover, if
$M$ is a right $C$-comodule, then the lattices
$\mathcal{L}_{Comod-C}(M)$ and
$\mathcal{L}_{\mathcal{M}^{C^1}}(M)$ of subobjects of
 $M$ in these two categories coincide.
\end{proposition}
\begin{proof}
(i) It is easy to check that the map $\overline{f}:D\rightarrow
C^1$ defined by $\overline{f} (d)=f(d)+\epsilon_D(d)e$ for any
$d\in D$ satisfies the required conditions.\\
(ii) If $M$ is a right $C$-comodule with comodule structure map
$\rho:M\rightarrow M\otimes C$, then $M$ can be viewed as a
counital right $C^1$-comodule with comodule structure map
$\rho^1:M\rightarrow M\otimes C^1$, $\rho^1(m)=\rho (m)+m\otimes
e$. By corestricting scalars for this $C^1$-comodule via $\pi$, we
obtain the initial $C$-comodule $M$.

Now if $M$ is a right $C^1$-comodule with comodule structure map
$m\mapsto \sum m_{[0]}\otimes m_{[1]}$, then it becomes a right
$C$-comodule via $\pi$, and further a right $C^1$-comodule with
comodule structure map $m\mapsto m\otimes e+\sum m_{[0]}\otimes
\pi(m_{[1]})$. If we write $m_{[1]}=m_{[1]}'+m_{[1]}'' e$, with
$m_{[1]}'\in C, m_{[1]}''\in \KK$, then $m_{[1]}'=\pi(m_{[1]})$
and $m=\sum m_{[1]}''m_{[0]}$ by the counit property. Therefore,
$m\otimes e+\sum m_{[0]}\otimes \pi(m_{[1]})=\sum
m_{[1]}''m_{[0]}\otimes e + \sum m_{[0]}\otimes m_{[1]}' =\sum
m_{[0]}\otimes m_{[1]}$, so we obtain the initial $C^1$-comodule
structure on $M$.

It is clear that the subobjects of $M$ in the two categories are
the same.
\end{proof}

\begin{corollary} \label{thfundcomodule}
Let $C$ be a coalgebra, and $(M,\rho)$ a right $C$-comodule. Then
the subcomodule generated by any element $m\in M$ is finite
dimensional.
\end{corollary}
\begin{proof}
The result is known for counital comodules over coalgebras with
counit (see \cite[Theorem 2.1.7]{DNR}). Now everything is clear
since
$\mathcal{L}_{Comod-C}(M)=\mathcal{L}_{\mathcal{M}^{C^1}}(M)$
shows that the $C$-subcomodule of $M$ generated by $m$ is the same
to the $C^1$-subcomodule of $M$ generated by $m$.
\end{proof}

\begin{proposition} \label{isoC*1}
Let $C$ be a coalgebra. Then $(C^*)^1\simeq (C^1)^*$ as algebras
with unit.
\end{proposition}
\begin{proof}
Let $\pi:C^1\rightarrow C$ be the natural projection and
$\pi^*:C^*\rightarrow (C^1)^*$ its dual map. We show that the pair
$((C^1)^*,\pi^*)$ satisfies the universal property of the
unitalization of the algebra $C^*$, and the desired isomorphism
follows. If $B$ is an algebra with unit $1_B$, and
$f:C^*\rightarrow B$ is an algebra map, let
 $\overline{f}:(C^1)^*\rightarrow B$ be defined by
$\overline{f}(\phi)=f(\phi_{|C})+\phi (e)1_B$ for any $\phi\in
(C^1)^*$, where $\phi_{|C}$ is the restriction of $\phi$ to $C$,
where we consider $C^1=C\oplus \KK e$. It is then easy to show
$\overline{f}$ is a map of algebras with identity satisfying
$\overline{f}\pi^*=f$, and that $\overline{f}$ is unique with
these properties.\\
Alternatively, one can show the isomorphism directly using the
construction of the unitalization and counitalization.
\end{proof}

\begin{corollary}
Let $C$ be a (not necessarily counital) coalgebra, and let $M$ be a right $C$-comodule. Then the lattices
$\mathcal{L}_{Comod-C}(M)$ and $\mathcal{L}_{C^*-Mod}(M)$
coincide.
\end{corollary}
\begin{proof}
We consider the functors
$$Comod-C\longrightarrow \mathcal{M}^{C^1}\longrightarrow
_{(C^1)^*}\mathcal{M}\longrightarrow
_{(C^*)^1}\mathcal{M}\longrightarrow _{C^*-Mod}\mathcal{M}$$ where
the first functor is the isomorphism of categories from
Proposition \ref{l.3} (ii), the second functor is the usual one
that regards comodules as modules over the dual algebra, the third
functor is the isomorphism of categories associated to the
isomorphism of algebras from Proposition \ref{isoC*1}, and the
fourth functor is the isomorphism of categories explained at the
beginning of this section for an arbitrary algebra. It is easy to
check that the composition of these functors is just the usual
functor $Comod-C\longrightarrow _{C^*-Mod}\mathcal{M}$. It follows
that
$$\mathcal{L}_{Comod-C}(M)=\mathcal{L}_{\mathcal{M}^{C^1}}(M)=\mathcal{L}_{_{(C^1)^*}\mathcal{M}}(M)
=\mathcal{L}_{_{(C^*)^1}\mathcal{M}}(M)=\mathcal{L}_{C^*-Mod}(M)$$

\end{proof}

Next we note that the result stating that the category of right
$C$-comodules is isomorphic to the category of rational left
$C^*$-modules holds in the non-counital case, too. The proof is as
in the case where $C$ has counit, see for example \cite[Theorem
2.2.5]{DNR}.

Now we can extend the fundamental theorem of coalgebras to the
non-counital case.

\begin{proposition}\label{th.fco}
Let $C$ be a coalgebra (not necessarily with counit). Then the
subcoalgebra generated by an element $c\in C$ is finite
dimensional.
\end{proposition}
\begin{proof}
A subcoalgebra of $C$ is the same as a sub-bicomodule of $C$.
Regarding $C$-bicomodules as $C\otimes C^{op}$-right comodules,
the statement follows from Corollary \ref{thfundcomodule}.
\end{proof}

\section{The finite dual of a non-unital algebra}\label{sectdualfinit}

 If $A$ is an algebra with
identity, then $A^0$ is the subspace of $A^*$ consisting of all
elements $f$ with the property that $ker(f)$ contains a two-sided
ideal of $A$ of finite codimension. This property of $f$ is
equivalent to the existence of two finite families $(g_j)_j,
(h_j)_j\in A^*$ such that $f(ab)=\sum\limits_jg_j(a)h_j(b)$ for
any $a,b\in A$. For such an $f$, the families $(g_j)_j$ and
$(h_j)_j$ can be chosen in $A^0$, and the mapping $f\mapsto
\sum\limits_jg_j\otimes h_j$ defines a coassociative
comultiplication on $A^0$. Thus $A^0$ is a coalgebra with counit
$f\mapsto f(1_A)$, called the finite dual of $A$.

In this section, we extend the construction of the finite dual of
an algebra to the non-unital case. Let $A$ be an arbitrary
algebra, $A^1$ its unitalization, and $i:A\rightarrow A^1$ the
inclusion map. Let $i^*:(A^1)^*\rightarrow A^*$ be the dual map.
We consider the finite dual $(A^1)^0$ of $A^1$, which has a
structure of a coalgebra with counit as above, with
comultiplication denoted by $\Delta$.

\begin{proposition} \label{constrdelta}
$i^*((A^1)^0)$ has a unique coalgebra structure making the
restriction of $i^*$ to $(A^1)^0$ a coalgebra map.
\end{proposition}
\begin{proof}
We show that there exists a unique linear map
$\delta:i^*((A^1)^0)\rightarrow i^*((A^1)^0)\otimes i^*((A^1)^0)$
making the following diagram commutative.


$$\xymatrix{
(A^1)^0 \ar[r]^{\Delta}\ar[d]_{i^*} & (A^1)^0\otimes (A^1)^0\ar[d]^{i^*\otimes i^*} \\
i^*((A^1)^0) \ar[r]^>>>>{\delta} & i^*((A^1)^0)\otimes
i^*((A^1)^0) }$$


Note that for simplicity we kept the notation $i^*$ for the
restriction of this map to $(A^1)^0$. Indeed, let $f\in (A^1)^0$
such that $i^*(f)=0$, i.e. the restriction $f_{|A}$ of $f$ to $A$
is zero. Let $\Delta(f)=\sum\limits_jg_j\otimes h_j$, thus
$f(zt)=\sum\limits_jg_j(z)h_j(t)$ for any $z,t\in A^1$. Then
$(i^*\otimes i^*)\Delta(f)=\sum\limits_jg_{j|A}\otimes h_{j|A}$.
If $\theta:A^*\otimes A^*\rightarrow (A\otimes A)^*$ is the
natural embedding, then $\theta (\sum\limits_jg_{j|A}\otimes
h_{j|A})(a\otimes b)=\sum\limits_jg_j(a)h_j(b)=f(ab)=0$. We obtain
that $(i^*\otimes i^*)\Delta(f)=0$, and so the existence and
uniqueness of $\delta$ follows.

Now we have that $$(\delta \otimes I)\delta i^*=(i^*\otimes
i^*\otimes i^*)(\Delta \otimes I)\Delta$$ and
$$(I\otimes \delta )\delta i^*=(i^*\otimes
i^*\otimes i^*)(I\otimes \Delta)\Delta$$ Since $\Delta$ is
coassociative and $i^*$ is an epimorphism, we must have $(\delta
\otimes I)\delta=(I\otimes \delta )\delta$.
\end{proof}

We can derive the following description of $i^*((A^1)^0)$.

\begin{proposition} \label{charfinitedual}
$i^*((A^1)^0)$ is the set of all elements $f\in A^*$ such that
$Ker(f)$ contains a two-sided ideal of $A$ of finite codimension.
\end{proposition}
\begin{proof}
If $f\in i^*((A^1)^0)$, then $f=g_{|A}$ for some $g\in (A^1)^0$.
Let $J$ be a two-sided ideal of $A^1$ of finite codimension such
that $J\subseteq Ker(g)$. Then $Ker(f)$ contains $A\cap J$, a
two-sided ideal of $A$ of finite codimension.

Conversely, if $Ker(f)$ contains a two-sided ideal $I$ of $A$ of
finite codimension, let $g\in (A^1)^*$ be a linear map such that
$g_{|A}=f$. Then $I$ is also a two-sided ideal of $A^1$, it has
finite codimension in $A^1$, and $I\subseteq Ker(g)$. Thus $g\in
(A^1)^0$ and $f=i^*(g)\in i^*((A^1)^0)$.
\end{proof}

We denote $i^*((A^1)^0)$ by $A^0$, and we call the coalgebra
$(A^0,\delta)$ the finite dual of $A$. Proposition
\ref{charfinitedual} and the description of $\delta$ in
Proposition \ref{constrdelta} shows that there is no inconsistency
in the notation, since in the case where $A$ is an algebra with
identity, $i^*((A^1)^0)$ is just the usual finite dual of the
algebra with identity $A$.

If $\phi:A\rightarrow B$ is an algebra map, then
$\phi^*(B^0)\subseteq A^0$. Indeed, this follows from the fact
that the inverse image of a finite codimensional two-sided ideal
of $B$ through $\phi$ is a finite codimensional two-sided ideal of
$A$. We denote by $\phi^0:B^0\rightarrow A^0$ the map induced by
$\phi^*$. The mappings $A\mapsto A^0$ and $\phi\mapsto \phi^0$
define a contravariant functor $(-)^0:alg\rightarrow coalg$, where
$alg$ is the category of algebras, and $coalg$ is the category of
coalgebras. As in the (co)unital case (see for example
\cite[Theorem 1.5.22]{DNR}) one sees that $(-)^0$ is a left
adjoint for the "dual algebra" functor $(-)^*:coalg\rightarrow
alg$.

Next we give several characterizations of the finite dual $A^0$.
Some of their proofs carry over from the unital case in a
straightforward way, but some other need  more attention.

 We first introduce a notation. If $\eta:A\rightarrow \End(V)$
is a finite dimensional representation of $A$, let $v_1,\dots,v_n$
be a basis of $V$ and let $\phi:End(V)\simeq M_n(\KK)$ be the
algebra isomorphism associated to this basis. Let $\rho=\phi
\eta:A\rightarrow M_n(\KK)$ be the resulting matrix
representation, and $\rho_{ij}$ be the coefficient functions of
$\rho$. Thus $\rho(a)=(\rho_{ij}(a))_{i,j}$ and $\eta
(a)(v_j)=\sum_i\rho_{ij}(a)v_i$ for any $a\in A$ and $1\leq
i,j\leq n$. Let $R(A)\subset A^*$ be the span of all possible such
coefficient functions $\rho_{ij}$, for arbitrary representations
$\eta$ and arbitrary choice of basis $(v_i)_i$. This is called the
set of representative functions on $A$. In the unital case, this
is a classical notion, with roots in representations of compact
groups, and of algebraic groups. For such groups, it is well known
now that the space of representative functions is the same as the
representative
 functions on an apropriate (group) algebra, and, for compact and algegraic
 groups (or group schemes), it forms a Hopf algebra. Now we can give the
equivalent characterizations of $A^0$; the equivalent statements
of the following proposition are well known for unital algebras.

\begin{proposition}
Let $A$ be an algebra (not necessarily with unit), and let $f:A\rightarrow \KK$ be a linear map. The following are equivalent.\\
(i) $f\in A^0$.\\
(ii) $\ker(f)$ contains a left ideal of finite codimension.\\
(iii) $\ker(f)$ contains a right ideal of finite codimension.\\
(iv) $Af$ is finite dimensional.\\
(v) $fA$ is finite dimensional.\\
(vi) The bimodule generated by $f$, $\KK f+Af+fA+AfA$, is finite
dimensional.\\ (vii) There exist $g_i,h_i\in A^*, i=1,\dots,n$
such that $f(ab)=\sum\limits_{i=1}^ng_i(a)h_i(b)$ for any $a,b\in
A$. \\
(viii) $f\in R(A)$.
\end{proposition}
\begin{proof}
The equivalence of the first seven conditions can be proved by
adapting the arguments from the unital case, see for example the
proof of \cite[Lemma 9.1.1]{mo}. One change is that the
$A$-submodule generated by $f$ in the left $A$-module $A^*$ is
$Af+\KK f$, which change does not affects the finiteness of the
dimension. Similarly $fA$ must be replaced by $fA+\KK f$, and
$AfA$ by $AfA+Af+fA+\KK f$ in the proof. Another difference is at
(ii)$\Rightarrow$(i), where we have to show that a left ideal $I$
of finite codimension contains a two-sided ideal of finite
codimension. For this, we take the finite dimensional left
$A$-module $A/I$, and the associated representation
$\phi:A\rightarrow End(A/I)$, $\phi(a)(\hat{b})=\widehat{ab}$,
where the hat indicates the class modulo $I$. Then it is easy to
check that $I\cap Ker(\phi)$ is a two-sided ideal of $A$ of finite
codimension. \\
(viii)$\Rightarrow$ (i) It is enough to show that any coefficient
functions $\rho_{ij}$ associated to a representation of $A$ lies
in $A^0$. Since $\rho(ab)=\rho (a)\rho(b)$, we obtain that
$\rho_{ij}(ab)=\sum \limits_{r}\rho_{ir}(a)\rho_{rj}(b)$ for any
$a,b\in A$, so $\rho_{ij}\in A^0$. \\
(i)$\Rightarrow$ (viii) Let $I$ be a two-sided ideal contained in
$\ker(f)$. For a linear map $h:A\rightarrow \KK$ with $I\subseteq
\ker(h)$ write $\overline{h}$ for the induced map to $A/I$. Let
$B=(A/I)^1$ be the unitalization of $A/I$, and let
$u:A/I\hookrightarrow B$ be the inclusion map. Let
$v_1,\ldots,v_n$ be a basis of $B$. Consider $B$ as a
representation of $A$, via $\eta:A\rightarrow End(B)$, and let
$\rho:A\rightarrow M_n(\KK)$ be the corresponding matrix
representation associated to the basis $v_1,\ldots,v_n$. Since
$1_B=\sum\limits_{i}\alpha_jv_j$ for some scalars $\alpha_j$, we
see that $\hat{a}=\sum\limits_{j}\alpha_j\hat{a}v_j=
\sum\limits_{j,i}\alpha_j\overline{\rho}_{ij}(\hat{a})v_i=
\sum\limits_i(\sum\limits_j\alpha_j\overline{\rho}_{ij}(\hat{a}))v_i$
for any $a\in A$, where $\hat{a}$ denotes the class of $a$ in
$A/I$. Since $\hat{a}=\sum\limits_iv_i^*(\hat{a})v_i$, where
$(v_i^*)_i$ is the dual basis of $(v_i)_i$, we obtain that
$v_i^*(\hat{a})=\sum\limits_j\overline{\rho}_{ij}(\hat{a})$, so
$u^*(v_i^*)=\sum\limits_j\overline{\rho}_{ij}$. As $u^*$ is
surjective, $(u^*(v_i^*))_i$ spans $(A/I)^*$, and then
$(\overline{\rho}_{ij})_{i,j}$ also spans $(A/I)^*$. Therefore
$\overline{f}=\sum\limits_{i,j}\beta_{ij}\overline{\rho}_{ij}$ for
some scalars $\beta_{ij}$, and this shows that
$f=\sum\limits_{i,j}\beta_{ij}\rho_{ij}$, since $I\subseteq
\ker(f)$ and $I\subseteq \ker(\rho)$.
 Hence, $f\in R(A)$.
\end{proof}

Thus, we note that $A^0=R(A)$ for every not necessarily unital
algebra. By the previous proposition, the coalgebra structure of
$R(A)$ can also be obtained as
$R(A)=\lim\limits_{\stackrel{\rightarrow}{I}}(A/I)^*$, with the
limit being taken over all cofinite ideals $I$.

The following shows that the unitalization, the counitalization
and the finite dual functors are compatible in some sense.

\begin{proposition}
Let $A$ be an algebra. Then $(A^1)^0\simeq (A^0)^1$ as coalgebras
with counit.
\end{proposition}
\begin{proof}
Let $i^0:(A^1)^0\rightarrow A^0$ be the image of the inclusion
$i:A\rightarrow A^1$ through the functor $(-)^0$. We show that
$((A^1)^0, i^0)$ satisfies the universal property of the
counitalization $((A^0)^1,\pi)$ of $A^0$; the desired isomorphism
(which moreover is compatible with $i^0$ and $\pi$) follows from
this.

Let $D$ be a coalgebra with counit, and let $f:D\rightarrow A^0$
be a map of coalgebras. Let $\overline{f}:D\rightarrow (A^1)^0$,
$\overline{f}(d)(a+\alpha u)=f(d)(a)+\alpha \varepsilon_D(d)$ for
any $d\in D, a\in A, \alpha \in \KK$. Since $f(d)\in A^0$, there
exists an ideal $I$ of $A$ of finite codimension with $I\subseteq
Ker(f(d))$. Then $I$ is also an ideal of finite codimension in
$A^1$ and $I\subseteq Ker(\overline{f}(d))$, and this shows that
indeed $\overline{f}(d)\in (A^1)^0$.

We now show that $\overline{f}$ is a coalgebra map. Let
$\gamma:(A^1)^0\otimes (A^1)^0\rightarrow (A^1\otimes A^1)^0$ be
the natural embedding. We have that

\bea (\gamma \Delta_{(A^1)^0}\overline{f}(d))((a+\alpha u)\otimes
(b+\beta u))&=&\sum \overline{f}(d)_1(a+\alpha
u)\overline{f}(d)_2(b+\beta u)\\
&=&\overline{f}(d)((a+\alpha u)(b+\beta u))\\
&=&\overline{f}(d)(ab+\alpha b+\beta a+\alpha\beta u)\\
&=&f(d)(ab+\alpha b+\beta a)+\alpha \beta \varepsilon_D(d) \eea
and

\bea (\gamma (\overline{f}\otimes
\overline{f})\Delta_D(d)((a+\alpha u)\otimes (b+\beta u))&=&\sum
\overline{f}(d_1)(a+\alpha
u)\overline{f}(d_2)(b+\beta u)\\
&=&\sum (f(d_1)(a)+\alpha \varepsilon_D(d_1))(f(d_2)(b)+\beta
\varepsilon_D(d_2))\\
&=&\beta f(d)(a)+\alpha f(d)(b)+\alpha\beta \varepsilon_D(d)+ \\
& & +\sum f(d_1)(a)f(d_2)(b)\eea

Since $\sum f(d_1)(a)f(d_2)(b)=\sum f(d)_1(a)f(d)_2(b)=f(d)(ab)$,
we obtain that $\gamma \Delta_{(A^1)^0}\overline{f}=\gamma
(\overline{f}\otimes \overline{f})\Delta_D$. Since $\gamma$ is
injective, this shows that $ \Delta_{(A^1)^0}\overline{f}=
(\overline{f}\otimes \overline{f})\Delta_D$, i.e. $\overline{f}$
is a coalgebra map.

It is clear that $i^0\overline{f}=f$. Moreover, if
$\tilde{f}:D\rightarrow (A^1)^0$ is another map of coalgebras with
counit such that $i^0\tilde{f}=f$, then for any $d\in D$ we have
that
$\tilde{f}(d)(u)=\varepsilon_{(A^1)^0}(\tilde{f}(d))=\varepsilon_D(d)$,
and
$\tilde{f}(d)(a)=\tilde{f}(d)(i(a))=(i^0\tilde{f}(d))(a)=f(d)(a)$
for any $a\in A$. We obtain that $\tilde{f}(d)(a+\alpha
u)=f(d)(a)+\alpha\varepsilon_D(d)=\overline{f}(d)(a+\alpha u)$, so
$\tilde{f}=\overline{f}$.

\end{proof}

Now we give an extension of the result saying that locally finite
unital representations of an algebra $A$ with identity are just
corepresentations of the finite dual $A^0$, see \cite[Chapter 3,
page 126]{abe}. If $A$ is an arbitrary algebra, a module $M$ is
called locally finite if the submodule generated by any element is
finite dimensional. We denote by $Locfin\,A-Mod$ the full
subcategory of $A-Mod$ whose objects are the locally finite
modules.

\begin{proposition} \label{isocategories}
Let $A$ be an algebra. Then the categories $Locfin\,A-Mod$ and
$Comod-A^0$ are isomorphic. In the case where $A$ is unital, this
isomorphism of categories restricts to an isomorphism between the
category of locally finite unital left $A$-modules and the
category of counital right $A^0$-comodules.
\end{proposition}
\begin{proof}
Let $M$ be a right $A^0$-comodule, with comodule structure given
by $m\mapsto \sum m_0\otimes m_1$. Then $M$ is a left $A$-module
with action given by $am=\sum m_1(a)m_0$, and this is locally
finite, since the $A$-submodule generated by $m$ is $Am+\KK m$,
which is finite dimensional since it is contained in the span of
$m$ and all $m_0$'s.

Conversely, if $M$ is a locally finite $A$-module, let $m\in M$,
and let $m_1,\ldots,m_n$ be a basis of $Am+\KK m$, the
$A$-submodule generated by $m$. Define $a_1^*,\ldots,a_n^*\in A^*$
by $am=\sum \limits_{i=1,n}a_i^*(a)m_i$ for any $a\in A$. Then it
is easy to check that $a_i^*\in A^0$ for any $i$, that
$\sum\limits_im_i\otimes a_i^*$ does not depend on the choice of
the basis $m_1,\ldots,m_n$, and the mapping $m\mapsto
\sum\limits_im_i\otimes a_i^*$ defines a right $A^0$-comodule
structure on $M$.

These correspondences are compatible with morphisms and they
define an isomorphism of categories. The last part is
straightforward.
\end{proof}

Let us consider now a bialgebra $H$, in the non-unital
non-counital sense. This means that $H$ is an algebra (not
necessarily with unit) and a coalgebra (not necessary with
counit), such that the comultiplication is an algebra morphism. If
moreover $H$ is a counital coalgebra and the comultiplication is
an algebra morphism, then $H$ is called a counital coalgebra.
Similarly  one defines unital bialgebras.

\begin{proposition} \label{finitedualbialgebra}
Let $H$ be a bialgebra. Then $H^0$, the finite dual of the
underlying algebra structure, is a subalgebra of $H^*$. Moreover,
$H^0$ is a bialgebra. If $H$ is counital (respectively unital),
then $H^0$ is unital (respectively counital).
\end{proposition}
\begin{proof}
Let $f,g\in H^0$, and let $\delta (f)=\sum f_1\otimes f_2$,
$\delta(g)=\sum g_1\otimes g_2$, where $\delta$ is the
comultiplication of $H^0$. Then for any $a,b\in H$ we have that
$(fg)(ab)=\sum f(a_1b_1)g(a_2b_2)=\sum
f_1(a_1)f_2(b_1)g_1(a_2)g_2(b_2)=\sum (f_1g_1)(a)(f_2g_2)(b)$, and
this shows that $fg\in H^0$ and $\delta (fg)=\sum f_1g_1\otimes
f_2g_2$. Thus $H^0$ is a bialgebra. The last part is obvious.
\end{proof}

\begin{examples}
(1) Let $S$ be a semigroup. Then the semigroup algebra $\KK S$ is
a counital bialgebra.  In fact, since $(\KK S)^*\simeq Func
(S,\KK)$, the algebra of functions on $S$ (with values in $\KK$),
then $(\KK S)^0$ is just the set $R_{\KK}(S)$ of representative
functions, i.e. the set of all functions $f:S\rightarrow \KK$ for
which there exist finite families of functions $(u_i)_i$ and
$(v_i)_i$ on $S$ such that $f(xy)=\sum_iu_i(x)v_i(y)$ for any
$x,y\in S$. By Proposition \ref{finitedualbialgebra}, $R_{\KK}(S)$
is a unital bialgebra.

(2) Let $\Gamma$ be a quiver and let $\KK [\Gamma]$ be the
associated quiver algebra, which is unital if and only if $\Gamma$
has finitely many vertices. It seems to be a difficult problem to
describe $\KK [\Gamma]^0$ explicitly for an arbitrary $\Gamma$.
Indeed, consider the case of $\Gamma$ consisting of one vertex $v$
and $n$ arrows ($n$ loops at $v$). In this situation,
$\KK[\Gamma]\cong \KK<X_1,\dots,X_n>$, the noncommutative algebra
of polynomials in $n$ variables. Determining the finite dual of
this algebra would involve the clasifying in some way the cofinite
ideals of $\KK<X_1,\dots,X_n>$, or equivalently, the annihilators
of finite dimensional representations. Classification of finite
dimensional $\KK<X_1,\dots,X_n>$-modules is however a ``wild"
problem, and it is to be expected that describing the finite dual
of this non-commutative polynomial algebra should be of similar
dificulty. In fact, it is known that $\KK<X_1,\dots,X_n>^0$ is the
cofree coalgebra over a finite dimensional vector space of
dimension $n$ (see \cite[Theorem 2.4.2]{abe}). More generally, if
the quiver has oriented cycles, the representation theory will
bear similarities to that of $K<X_1,\dots, X_n>$. In particular,
if at least to different such cycles exist, the category of finite
dimensional representations is wild.

Under certain conditions or in particular cases the description is
known. For instance, it is proved in \cite[Theorem 3.3]{DIN1} that
if $\Gamma$ has no oriented cycles and there are only finitely
many arrows between any two vertices, then $\KK [\Gamma]^0$ is
isomorphic to the path coalgebra $\KK \Gamma$ associated to
$\Gamma$.  Another interesting situation is when $\Gamma$ has just
one vertex and just one arrow (a loop), in which case $\KK
[\Gamma]$ is just the polynomial algebra $\KK [X]$. In this case,
$\KK [\Gamma]^0$ is identified with the linearly recursive
functions on $\KK [X]$, see \cite[Example 9.1.7]{mo}. It can
precisely be described as

$$K[X]^0=\bigoplus\limits_{f{\rm\,irreducible}}[\lim\limits_{\stackrel{\longrightarrow}{n\in
\ZZ_{\geq 0}}}(K[X]/(f^n))^*]$$

(3) Let $(X,\leq)$ be a partially ordered set which is locally
finite, i.e., the set $\{z|x\leq z\leq y\}$ is finite for any
$x\leq y$ in $X$. Let $\KK X$ be the incidence coalgebra of $X$.
This is
 the vector space with basis $\{ e_{x,y}|x,y\in
X, x\leq y\}$, comultiplication  defined by
$\Delta(e_{x,y})=\sum_{x\leq z\leq y}e_{x,z}\otimes e_{z,y}$, and
counit defined by $\epsilon (e_{x,y})=\delta_{x,y}$ for any
$x,y\in X$ with $x\leq y$. The dual algebra of $\KK X$ is
isomorphic to the incidence algebra $IA(X)$, which is the space of
all functions $f:\{(x,y)|x,y\in X, x\leq y\}\rightarrow K$, with
multiplication given by $(fg)(x,y)=\sum_{x\leq z\leq
y}f(x,z)g(z,y)$ for any $f,g\in IA(X)$ and any $x,y\in X, x\leq
y$. Let  $FIA(X)$ be the set of all elements of $IA(X)$ of finite
support. Then $FIA(X)$ is a subalgebra of $IA(X)$, non-unital when
$X$ is infinite. It is proved in \cite[Theorem 4.2]{DIN1} that
$FIA(X)^0\simeq KX$ as coalgebras.

\end{examples}

\section{Applications to coalgebras}

We give an application of the above constructions to coalgebras.
We note a connection between semiperfect coalgebras and notions of
coreflexivity for coalgebras. In this section $C$ will be a
coalgebra with counit. Then there is a decomposition into
indecomposable left comodules $C=\bigoplus\limits_i E(S_i)$, where
$S_i$ are simple left $C$-comodules and $E(S_i)$ is an injective
hull of $S_i$ contained in $C$. A similar decomposition holds on
the right: $C=\bigoplus\limits_jE(T_j)$. Recall from \cite{L} that
a coalgebra is left (right) semiperfect if the right (left)
injective indecomposable comodules $E(T_j)$ (respectively,
$E(S_i)$) are finite dimensional. We recall from \cite[Theorem
3.2.3]{DNR} that the following are equivalent for a coalgebra $C$:
(i) $C$ is right semiperfect; (ii) The category $\Mm^C$ has enough
projectives; (iii) $Rat({}_{C^*}C^*)$ is dense in $C^*$ in the
finite topology of $C^*$; (iv) Any finite dimensional right
$C$-comodule has a projective cover.

We also recall from \cite{HR,R,Rad} that a coalgebra is called
coreflexive if the canonical map $C\rightarrow (C^*)^0$ is
bijective (it is always injective). Given a coalgebra $C$, besides
the dual algebra $C^*$, there are also two versions of dual
(co)modules, namely $Rat(_{C^*}C^*)$, the left rational dual
module (which is a right comodule), and $Rat(C^*_{C^*})$, the
right rational dual module. These have been considered before in
relation to self-duality properties of coalgebras and Hopf
algebras (see for example, \cite{I,I1,Su,Sw,Sw1}), and play an
important role in the theory of Hopf algebras with non-zero
integral. These duals are ideals of $C^*$, so they can also be
regarded as non-unital algebras. We can define a map
$\phi_l:C\rightarrow (Rat({}_{C^*}C^*))^0$ in the natural way by
$\phi_l(c)(c^*)=c^*(c)$. We note that $Ker\, \phi_l(c)=c^\perp
\cap Rat(_{C^*}C^*)$ has finite codimension in $Rat(_{C^*}C^*)$,
since it contains $D^\perp \cap Rat(_{C^*}C^*)$, where $D$ is the
subcoalgebra of $C$ generated by $c$. Since $D$ is finite
dimensional, then $D^\perp$ has finite codimension in  $C^*$, and
then $D^\perp \cap Rat(_{C^*}C^*)$ has finite codimension in
$Rat(_{C^*}C^*)$. Thus $\phi_l(c)$ lies indeed in
$(Rat({}_{C^*}C^*))^0$. It is easy to see that $\phi_l$ is a
morphism of coalgebras. We introduce the following

\begin{definition}
Let $C$ be a coalgebra. We call $C$ left coreflexive if the map
$\phi_l$ is bijective.
\end{definition}

\begin{proposition}
If $C$ is left coreflexive, then any finite dimensional left
$Rat(_{C^*}C^*)$-module $M$ has a structure of a right
$C$-comodule such that the associated left $C^*$-module structure
gives by restriction of scalars the initial
$Rat(_{C^*}C^*)$-module structure of $M$.
\end{proposition}
\begin{proof}
Since $M$ has finite dimension, it is locally finite, and then by
 Proposition \ref{isocategories}, $M$ is a right $Rat({}_{C^*}C^*)^0$-comodule
with coaction $m\mapsto \sum m_{[0]}\otimes m_{[1]}$ such that the
$Rat(_{C^*}C^*)$-action on $M$ is given by $c^*m=\sum
m_{[1]}(c^*)m_{[0]}$ for any $c^*\in Rat({}_{C^*}C^*)$ and $m\in
M$. Then $M$ becomes a right $C$-comodule (not necessarily
counital) via the coalgebra isomorphism $\phi_l$, i.e. the
comodule structure is given by $m\mapsto \sum m_0\otimes m_1=\sum
m_{[0]}\otimes \phi_l^{-1}(m_{[1]})$. This right $C$-comodule
structure induces a new left $C^*$-module structure on $M$, with
action denoted by $c^*\cdot m$ for $c^*\in C^*$ and $m\in M$. We
have that $c^*\cdot m=\sum c^*(m_1)m_0=\sum
c^*(\phi_l^{-1}(m_{[1]}))m_{[0]}$ for any $c^*\in C^*$ and $m\in
M$. In the case where $c^*\in Rat({}_{C^*}C^*)$ we have that
$c^*(\phi_l^{-1}(m_{[1]}))=\phi_l(\phi_l^{-1}(m_{[1]}))(c^*)=m_{[1]}(c^*)$,
so then $c^*\cdot m=\sum c^*(\phi_l^{-1}(m_{[1]}))m_{[0]}=\sum
m_{[1]}(c^*)m_{[0]}=c^*m$. Thus restricting this $C^*$-action on
$M$ to $Rat({}_{C^*}C^*)$ gives the initial action.
\end{proof}

The following explains and extends results of \cite{DIN1} about
path coalgebras and incidence coalgebras.

\begin{theorem} \label{maintheorem}
Let $C$ be a coalgebra. Then the following assertions are true.\\
(i) $C$ is right semiperfect if and only if $\phi_l$ is injective. In particular if $C$ is left coreflexive, then $C$ is right semiperfect.\\
(ii) If $C$ is right semiperfect and the coradical $C_0$ of $C$ is
coreflexive, then $C$ is left coreflexive.
\end{theorem}
\begin{proof}
(i) The map $\phi_l$ is injective if and only if $Ker\,
\phi_l=Rat({}_{C^*}C^*)^\perp=\{c\in C| c^*(c)=0,\,\forall c^*\in
Rat({}_{C^*}C^*)\}=0$,
which means that $Rat({}_{C^*}C^*)$ is dense in $C^*$ (see, for example, \cite[Corollary 1.2.9]{DNR}), i.e. $C$ is right semiperfect.\\
(ii) 
It remains to prove the surjectivity of $\phi_l$. Since $C_0$ is
coreflexive and $C$ is right semiperfect, by
\cite[Corollary 4.10]{I2} we see that $C$ is coreflexive. Let
$C=\bigoplus\limits_{j\in J}E(S_j)$ be the decomposition into
indecomposable left comodules as before. Since $C$ is right
semiperfect we have that each $E(S_j)$ is finite dimensional. Then
$C^*\simeq\prod\limits_jE(S_j)^*$ as left $C^*$-modules. In fact
we identify $C^*$ with $\prod\limits_jE(S_j)^*$, by regarding
$E(S_j)^*$ as the set of elements $c^*\in C^*$ such that
$c^*(E(S_p))=0$ for any $p\neq j$. For any $j\in J$ we denote by
$e_j$ the element of $C^*$ which agrees with $\epsilon$ on
$E(S_j)$ and vanishes on any $E(S_p)$ with $p\neq j$. Then
$E(S_j)^*=C^*e_j$ for any $j$, and $(e_j)_{j\in J}$ is a set of
orthogonal idempotents in $C^*$.

Denote $R=Rat({}_{C^*}C^*)$; we have that
$\bigoplus\limits_jE(S_j)^*\subseteq R$. Let $f\in R^0$, and let
$I$ be a cofinite two-sided ideal of $R$ such that $I\subseteq
\ker(f)$. The right $C^*$-module $R/I$ is finite dimensional, so
it is rational, since $C$ is coreflexive. Let $\rho:R/I\rightarrow
C\otimes R/I$ be a left coaction compatible with the right
$C^*$-module structure. The coalgebra of coefficients $D=cf(R/I)$
of $R/I$ is finite dimensional, so there is a finite set $F\subset
J$ such that $D\subseteq \bigoplus\limits_{j\in F}E(S_j)$. Let
$e=\varepsilon -\sum\limits_{j\in F}e_j$. Note that $R/I\cdot
e=0$, so $Re\subseteq I$.

Now $f_{|\oplus_{j\in F}E(S_j)^*}\in (\oplus_{j\in
F}E(S_j)^*)^*\simeq (\oplus_{j\in F}E(S_j))^{**}$. Since
$$\gamma:\oplus_{j\in F}E(S_j)\ra (\oplus_{j\in F}E(S_j))^{**},
\gamma(c)(c^*)=c^*(c)$$ is an isomorphism, we obtain that there is
$x\in \oplus_{j\in F}E(S_j)$ such that $f_{|\oplus_{j\in
F}E(S_j)^*}=\gamma(x)$, i.e. $f(c^*)=c^*(x)$ for any $c^*\in
\oplus_{j\in F}E(S_j)^*$.


We show that $f(u)=u(x)$ for all $u \in R$. Let
$g=\sum\limits_{j\in F}e_j$. Then  $u=ug+ue$, $ug\in \oplus_{j\in
F}E(S_j)^*$ and $ue\in Re\subseteq I$. Thus $f(ug)=(ug)(x)$,
$f(ue)=0$ and $f(u)=(ug)(x)$. On the other hand $\delta(x)=\sum
x_1\otimes x_2\in C\otimes \oplus_{j\in F}E(S_j)$, and we have
$(ue)(x)=\sum u(x_1)e(x_2)=0$. Therefore, we get
$f(u)=(ug)(x)=(ug)(x)+(ue)(x)=(ug+ue)(x)=u(x)$, and this ends the
proof.
\end{proof}

\begin{corollary}\label{c.main}
Let $C$ be a coalgebra such that the coradical $C_0$ is
coreflexive. Then $C$ is left coreflexive if and only if $C$ is
right semiperfect. In this case, $C$ is also coreflexive.
\end{corollary}
\begin{proof}
The equivalence follows from the previous Theorem
\ref{maintheorem}. If $C_0$ is coreflexive and $C$ is right
semiperfect, then $C$ is coreflexive by \cite[Corollary 4.10]{I2}.
\end{proof}

\begin{example}
It is possible that a coalgebra $C$ is coreflexive (so then its
coradical is also coreflexive, as a subcoalgebra of a coreflexive
coalgebra, see \cite[Proposition 3.1.4]{HR}), but $C$ is not left
coreflexive. Indeed, let $\Gamma$ be the infinite line quiver,
i.e. the vertices of $\Gamma$ are the integers, and there is an
arrow from $n$ to $n+1$ for any integer $n$. Then $C_0$ is the
grouplike coalgebra of the set $\mathbb Z$ of integers, so it is
coreflexive by \cite{HR} (see the comments after this example). By
\cite[Proposition 5.4]{DIN1}, $C$ is also coreflexive. On the
other hand, for any integer $n$ there exist infinitely many paths
starting from $n$ and infinitely many paths ending at $n$, so $C$
is neither left semiperfect nor right semiperfect, see
\cite[Corollary 6.3]{chin}. Therefore $C$ is not left coreflexive.
\end{example}

We note that the condition that $C_0$ is coreflexive in Theorem
\ref{maintheorem}(ii) is not a restrictive condition. Using the
terminology of \cite[Section 3.7]{HR}, a set $X$ is called
reasonable if every ultrafilter on $X$ closed under countable
intersections is principal. Such sets are called non-measurable
in set theoretic terminology. Countable sets are reasonable, power
sets and subsets of reasonable sets are reasonable, and unions and
direct products of families of reasonable sets, indexed by
reasonable sets, are also reasonable. Thus the class of reasonable
sets is very large, essentially every set we work with is
reasonable. It is proved in \cite[Theorem 3.7.5]{HR} that if the
basefield is infinite and the set of simple subcoalgebras of $C$
is reasonable, then $C_0$ is coreflexive.

We see by Corollary \ref{c.main} that if the set of simple
subcoalgebras of $C$ is reasonable and $k$ is infinite, then $C$
is left coreflexive if and only if $C$ is right semiperfect, thus
the two concepts coincide. Under the same conditions, we also have
that if $C$ is left coreflexive, then $C$ is coreflexive.

In the case where $C$ is left and right semiperfect, the
assumption on $C_0$ in Corollary \ref{c.main}, or on the set
of simple subcoalgebras of $C$ being reasonable in the discussion
above are no longer necessary.

\begin{theorem}\label{t.main2}
Let $C$ be a left and right semiperfect coalgebra. Then $C$ is
left coreflexive. Moreover, $\phi_l$ is an isomorphism of counital
coalgebras.
\end{theorem}
\begin{proof}
 Since $C$ is left and right semiperfect, we
have that
$R=Rat({}_{C^*}C^*)=Rat(C^*_{C^*})=\bigoplus\limits_{j\in
J}E(S_j)^*$, and this is an algebra with enough idempotents
$(e_j)_{j\in J}$ (see \cite[Section 3.3]{DNR}), where $e_j$ is
defined as in the proof of Theorem \ref{maintheorem}.

Let $f\in R^0$ and $I$ an ideal of finite codimension in $R$ such
that $I\subseteq Ker(f)$. By \cite[Lemma 2.1]{DIN1}, only finitely
many of the idempotents $(e_j)_{j\in J}$ can lie outside $I$. Then
there exists a finite subset $F$ of $J$ such that
$\bigoplus\limits_{j\in J\setminus F}E(S_j)^*\subseteq I$. As in
the proof of Theorem \ref{maintheorem}(ii), there exists $x\in
\oplus_{j\in F}E(S_j)$ such that  $f(c^*)=c^*(x)$ for any $c^*\in
\oplus_{j\in F}E(S_j)^*$. Now if $c^*\in \oplus_{j\in J\setminus
F}E(S_j)^*$, then $c^*\in I$, so $f(c^*)=0$, and $c^*(x)=0$ since
$x\in \oplus_{j\in F}E(S_j)$. Thus $f(c^*)=c^*(x)$, so
$f=\phi_l(x)$. This shows that $\phi_l$ is bijective.

By \cite[Proposition 2.4]{DIN1} the coalgebra $R^0$ has counit $E$
defined by $E(f)=\sum_jf(e_j)$. Then
$(E\phi_l)(c)=\sum_j\phi_l(c)(e_j)=\sum_je_j(c)=\varepsilon (c)$,
so $\phi_l$ is a morphism of counital coalgebras.

\end{proof}

\subsection{Examples in Hopf algebras}

{Before applying the above result to Hopf algebras, in order to motivate it
we recall here a class of coreflexive Hopf algebras considered in the introduction.
Let $G$ be an algebraic group over $\CC$, and let $L$ be its Lie algebra. Let
$H$ be the algebra of functions on $G$.

\begin{proposition}\label{p.algr}
With the above notations, if $L$ is a finite dimensional Lie
algebra (in particular in the case where $G$ is an affine
algebraic group), then the Hopf algebras $U(L)$ and $H^0$ are
coreflexive.
\end{proposition}
\begin{proof}
If $L$ is a finite dimensional Lie algebra over $\CC$, by \cite{HR}, $U(L)$
is coreflexive, since the coradical of $U(L)$ is 1-dimensional
and the space of primitives is finite dimensional (see also the results
of \cite[Sections 2 \& 4]{I2}). With the notations for $G$ and $H$ as above,
since $H^0=\CC[G]\otimes U(L)$ as coalgebras. To show that $H^0$ is coreflexive, one
can notice that $H^0$ is cocommutative, and the space of primitives is finite dimensional
(and the set of grouplikes is $G$); thus, \cite[5.1.3 Corollary]{HR} applies and
$H^0$ is coreflexive. Alternatively, one can use \cite[Theorem 3.4.3]{HR} to get
that $U(L)$ is strongly coreflexive, and since $\CC[G]$ is coreflexive, by
\cite[Theorem 4.4]{Rad} $\CC[G]\otimes U(L)$ is coreflexive. \\
It is also possible to check the coreflexivity of $H^0$ directly, using the fact
that $H^0=\CC[G]\otimes U(L)=\bigoplus\limits_{G}U(L)$ as coalgebras.
\end{proof}
}

Since a Hopf algebra with non-zero integrals is left and right
semiperfect as a coalgebra, we obtain as a consequence of Theorem
\ref{t.main2} the following.

\begin{corollary}
Let $H$ be a Hopf algebra with non-zero integrals. Then $H\simeq
(H^{*rat})^0$ as coalgebras.
\end{corollary}

Let $H$ be a Hopf algebra with non-zero left integral $t$ on $H$
(in particular $H$ can be any compact quantum group). Then
$H^{*rat}=H\rightharpoonup t$, where $\rightharpoonup$ denotes the
standard left $H$-action on $H^*$, induced by the right $H$-module
structure of $H$ (given by the multiplication of $H$). Then the
above result shows that the coalgebra structure of $H$ can be
reconstructed from the algebra structure of $H$ and a non-zero
integral on $H$. Thus knowing a non-zero integral is quite a
strong information about the structure of $H$.

We also note that, by the discussion above, any Hopf algebra with
non-zero integral is coreflexive as a coalgebra, provided that $k$
is infinite and the set of simple subcoalgebras of $H$ is
reasonable.  This essentially covers all usual examples.


\bigskip\bigskip\bigskip

\begin{center}
\sc Acknowledgment
\end{center}
The research  was supported by the UEFISCDI Grant
PN-II-ID-PCE-2011-3-0635, contract no. 253/5.10.2011 of CNCSIS.
The research of the second author was also supported by  the
strategic grant POSDRU/89/1.5/S/58852, Project ``Postdoctoral
program for training scientific researchers'' cofinanced by the
European Social Fund within the Sectorial Operational Program
Human Resources Development 2007-2013.


\bigskip\bigskip\bigskip



\begin{thebibliography}{99}


\bibitem{abe}
E.~Abe, Hopf Algebras, Cambridge University Press, Cambridge,
1977.

\bibitem{chin}
W.~Chin, M.~Kleiner, D.~Quinn, \textit{Almost split sequences for
comodules}, J. Algebra {\bf 249} (2002), 1-19.

\bibitem{C}
P. Cartier, \emph{A Primer of Hopf Algebras}, Lecture Notes IHES,
http://preprints.ihes.fr/2006/M/M-06-40.pdf


\bibitem{DNR} S.~D\u asc\u alescu, C.~N\u ast\u asescu, \c S.~Raianu, Hopf algebras.
An introduction, Marcel Dekker, New York, 2001.


\bibitem{DIN1} S.~D\u asc\u alescu, M.C.~Iovanov, C.~N\u ast\u
asescu, \textit{Quiver algebras, path coalgebras, and
coreflexivity},  Pac. J. Math. {\bf 262} (2013), 49-79.



\bibitem{HR}
R.~Heyneman, D.E.~Radford, \textit{Reflexivity and Coalgebras of
Finite Type}, J. Algebra {\bf 28} (1974), 215-246.

\bibitem{I}
M.C. Iovanov, \emph{Co-Frobenius Coalgebras}, J. Algebra 303
(2006), no. 1, 146--153.

\bibitem{I1}
M.C.Iovanov, \textit{Generalized Frobenius Algebras and Hopf
Algebras}, preprint arXiv, to appear, Can. J. Math.

\bibitem{I2}
M.C.Iovanov, \textit{On Extensions of Rational Modules}, preprint
arxiv, to appear, Israel J. Math. 


\bibitem{K}
C. Kassel, \emph{Quantum Groups}, Graduate Texts in Mathematics,
Springer 1995.

\bibitem{L}
B.I-Peng Lin, \emph{Semiperfect coalgebras}, J. Algebra {\bf 30}
(1974), 559-601.

\bibitem{mo} S.~Montgomery,  Hopf algebras and their actions
on rings, {\sl CBMS Reg. Conf. Series} {\bf 82}, American
Mathematical Society, Providence, 1993.

\bibitem{R}
D.E.~Radford, \textit{On the Structure of Ideals of the Dual
Algebra of a Coalgebra}, Trans. Amer. Math. Soc. Vol. 198 (Oct
1974), 123-137.

\bibitem{Rad}
D.E.~Radford, \textit{Coreflexive coalgebras}, J. Algebra {\bf 26}
(1973), 512--535.



\bibitem{Su}
J.B. Sullivan, \emph{The uniqueness of integrals for Hopf algebras
and some existence theorems of integrals for commutative Hopf
algebras}, J. Algebra 19 (1971), 426–-440.

\bibitem{Sw} M.E.~Sweedler, Hopf Algebras, Benjamin, New York, 1969.

\bibitem{Sw1}
M.E. Sweedler, \emph{Integrals for Hopf algebras}, Ann. Math 89
(1969), 323--335.

\bibitem{T1}
E.J.~Taft, \textit{Reflexivity of Algebras and Coalgebras},
American Journal of Mathematics, Vol. 94, No. 4 (Oct 1972),
1111-1130.

\bibitem{T2}
E.J.~Taft, \textit{Reflexivity of algebras and coalgebras. II.}
Comm. Algebra  {\bf 5}  (1977), no. 14, 1549--1560.




\end{thebibliography}
\end{document}